\documentclass[oneside]{amsart}
\usepackage{amsmath,amssymb,amsfonts,amsthm} \usepackage{color}
\usepackage{geometry} 
\usepackage{graphicx} \usepackage{enumitem} \usepackage[utf8]{inputenc}

\title[Geodesic invariant functions and applications to Landsberg surfaces]
{The geometry of geodesic invariant functions and applications to Landsberg
  surfaces}

\author[Elgendi]{S.G.~Elgendi}

\address{Salah G.~Elgendi, Department of Mathematics, Faculty of Science,
  Islamic University of Madinah, Saudi Arabia}

\email{salah.ali@fsc.bu.edu.eg}
  
\author[Muzsnay]{Z.~Muzsnay}

\address{Zolt\'an Muzsnay, Department of Geometry, Institute of
  Mathematics, University of Debrecen, Debrecen, Hungary}

\email{muzsnay@science.unideb.hu}



\def\ok#1{\textcolor[rgb]{0.0,0.0,0.0}{#1}}

\def\v{m_{\scalebox{0.6}{$S$}}}
\def\smallv{m_{\scalebox{0.4}{$S$}}}
\newcommand{\R}{{\mathbb R}}

\newcommand{\F}{{\mathcal F}}

\newcommand{\T}{{\mathcal T}}
\newcommand{\C}{{\mathcal C}}
\newcommand{\V}{{\mathcal V}}

\newcommand{\halmaz}[1]{\left\{#1\right\}}

\newcommand{\set}[1]{\left\{#1\right\}}

\newcommand{\im}{\operatorname{Im}}
\newcommand{\Ker}{\operatorname{Ker}}

\newcommand{\tm}{\T M}

\newcommand{\TM}{\mathcal T\hspace{-1pt}M}
\newcommand{\hp}{h_{_{\P}}}
\newcommand{\vp}{v_{_{\P}}}
\newcommand{\Hp}{\mathcal H_{_{\P}}}
\newcommand{\Vp}{\mathcal V_{_{\P}}}
\newcommand{\chol}{C^\infty_{{\mathcal{H}ol}}}
\newcommand{\hol}{{\mathcal{H}\hspace{-1pt}ol}}

\def\H{{\mathcal H}}
\def\V{{\mathcal V}}
\def\Span#1{\textrm{Span}{\left\{{#1}\right\}}}

\newcommand{\lie}[1]{{\mathcal L}_{#1}}
\newcommand*{\medcap}{\mathbin{\scalebox{1.5}{\ensuremath{\cap}}}}%

\def\P{\mathcal P}

\def\pa{\partial}
\def\paa{\dot{\partial}}

\numberwithin{equation}{section} 
\numberwithin{figure}{section} 

\theoremstyle{plain}
\newtheorem*{theorem*}{Theorem}
\newtheorem{theorem}{Theorem}[section]
\newtheorem{proposition}[theorem]{Proposition}
\newtheorem{lemma}[theorem]{Lemma}
\newtheorem{corollary}[theorem]{Corollary}
\newtheorem{definition}[theorem]{Definition}

\theoremstyle{definition}

\newtheorem{remark}[theorem]{Remark}
\theoremstyle{remark}

\newtheorem*{acknowledgement*}{Acknowledgement}

\begin{document}

\maketitle

\begin{abstract}
 In this paper, for a given spray $S$ on an $n$-dimensional manifold $M$,
  we investigate the geometry of $S$-invariant functions. For an
  $S$-invariant function $\P$, we associate a vertical subdistribution
  $\V_\P$ and find the relation between the holonomy distribution and
  $\V_\P$ by showing that the vertical part of the holonomy distribution is
  the intersection of \ok{all spaces $\V_{\F_S}$ associated to $\F_S$
    where $\F_S$} is the set of all Finsler functions that have the
  geodesic spray $S$.  As an application, we study the Landsberg Finsler
  surfaces. We prove that a Landsberg surface with $S$-invariant flag
  curvature is Riemannian or has a vanishing flag curvature. We show that
  for Landsberg surfaces with non-vanishing flag curvature, the flag
  curvature is $S$-invariant if and only if it is constant, in this case,
  the surface is Riemannian.  Finally, for a Berwald surface,   we prove that
   the flag curvature is $H$-invariant if and only if it is constant.
\end{abstract}

\section{Introduction}

A system of second order homogeneous ordinary differential equations
(SODE), whose coefficients do not depend explicitly on time, can be
identified with a special vector field, called spray.  The solution of the
SODE is called the geodesic of the spray.  The spray corresponding to the
geodesic equation of a Riemannian or Finslerian metric is called the
geodesic spray of the corresponding metric.

\ok{The concept of geodesic invariant functions (or equivalently
  $S$-invariant functions or first integrals of $S$) has various
  applications not only in Finsler and Riemann geometries, but also in
  physics. For example, the norm and the energy functions are geodesic
  invariant functions on Finslerian or Riemannian manifolds, on Landsberg
  surfaces the main scalar of the surface is $S$-invariant.  Also, in
  physics, if a geodesic invariant function is given, then this function
  can treated as a constant of motion, in other words, these functions are
  conserved along motion. Geodesic invariant functions can give important
  information on the geometric structure. See for example
  \cite{BucataruConstCret, Foulon} and references therein.}

By \cite{MZ_ELQ}, for a given spray $S$ on a $n$-dimensional manifold $M$,
we can associate the so-called holonomy distribution which is generated by
the \ok{horizontal vector fields} and their successive Lie
brackets. The functions on $TM$ that invariant with respect to the parallel
translation are called holonomy invariant functions. These functions are
constant along the holonomy distribution \cite{MZ-Elgendi-1}. It is easy to
see that the holonomy invariant functions are also $S$-invariant functions,
that is, constant along the spray.  However, the converse is not true, that
is not any function constant along the spray is holonomy invariant. In the
literature $S$-invariant functions are also known as first integrals of the
spray $S$, for example, we refer to \cite{BucataruConstCret,Foulon}.

In this paper, we investigate the geometry of distributions associate to
homogeneous $S$-invariant functions of degree $k\neq 0$.  \ok{A function
  $\P$ defined on $TM$ is called $k$-homogeneous, if it satisfies the
  equation $\P(\lambda v) = \lambda^k \P(v)$ for any $v \in TM$. We show
  that to any $k$-homogeneous $S$-invariant nontrivial function $\P$} one
can associate \ok{the} decomposition of $TTM$
\begin{equation}
  \label{eq:TTM_decomposition}
  TT M=\Hp\oplus Span\{S\}\oplus\Vp\oplus Span\{\C\}.    
\end{equation}
where $\Hp$ and $\Vp$ are $n-1$-dimensional sub-distribution of the
horizontal (resp. the vertical) spaces associated to the spray.  Moreover,
if $\P$ is a holonomy invariant function, then
\begin{equation}
  \label{eq:ker_dP_bis}
  \mathrm{Ker} \,  d\P =  \H \oplus\Vp,
\end{equation}
where $\H$ is the horizontal distribution associated to $S$.

As a special case, for a Finsler manifold $(M,F)$, since $F$ is constant
along its geodesic spray $S$ and also along the horizontal distribution
$\H$, we focus our attention to the distribution $\V_F$.  In
\cite{MZ-Elgendi-1}, the notion of metrizability freedom of sprays was
introduced.  For a given spray $S$, $\v$ shows how many essentially
different Finsler functions can be associated to it.  The metrizability
freedom of a spray can determined by the help of its holonomy distribution
$\hol$.  We prove that $\V_{\hol}$ and $\V_F$ coincide if and only if
the metrizability freedom of $S$ is one. In the case when $\v\geq 1$, then
$\V_{\hol}$ is a sub-distribution of $\V_F$ and we prove that
\begin{displaymath}
  \V_{\hol}=\underset{F \in \mathcal F_S}{\medcap} \V_{F}
\end{displaymath}
where $\mathcal F_S$ denotes the set of Finsler functions associated to the
spray $S$.

As an application, we turn our attention to the Landsberg surfaces.  We
show that for a Landsberg surface, if the flag curvature is $S$-invariant,
then the surface is Riemannian or has a vanishing flag curvature.  Also,
for a Landsberg surface with non-vanishing flag curvature $K$, then we
establish that $K$ is $S$-invariant if and only if $K$ is constant. In this
case, the surface is Riemannian.  Finally, we prove that, for a Berwald
surface, the flag curvature is $H$-invariant if and only if $K$ is
constant.

\section{Preliminaries}

\ok{$M$ is an $n$-dimensional smooth manifold, its tangent bundle
  $(TM,\pi_M,M)$, and its subbundle of nonzero tangent vectors
  $(\T M,\pi,M)$.  On the base manifold $M$, we indicate local coordinates
  by $(x^i)$, while on $TM$, the induced coordinates are    
  $(x^i, y^i)$.  The natural almost-tangent structure of $TM$ is defined
  locally by $J = \frac{\partial}{\partial y^i} \otimes dx^i$, which is the
  vector $1$-form $J$ on $TM$. The canonical or Liouville vector field is
  the vertical vector field $\C=y^i\frac{\partial}{\partial y^i}$ on $TM$.}

\bigskip

\subsection{Spray and Finsler manifold} \

\medskip

\noindent
\ok{The geometry of sprays and Finsler manifolds has a vast literature.
  Here we are using essentially the results and the terminology of
  \cite{Grifone_1972, Grifone_Muzsnay_2000}.}

A vector field $S\in \mathfrak{X}(\T M)$ is called a spray if $JS = \C$ and
$[\C, S] = S$. Locally, a spray is expressed as follows
\begin{equation}
  \label{eq:spray}
  S = y^i \frac{\partial}{\partial x^i} - 2G^i\frac{\partial}{\partial y^i},
\end{equation}
where the \emph{spray coefficients} $G^i=G^i(x,y)$ are $2$-homogeneous
functions in the $y=(y^1, \dots , y^n)$ variable. A curve
$\sigma : I \rightarrow M$ is called regular if
$\sigma' : I \rightarrow \tm$, where $\sigma'$ is the tangent lift of
$\sigma$. A regular curve $\sigma $ on $M$ is called \emph{geodesic} of a
spray $S$ if $S \circ \sigma' = \sigma''$. Locally, $\sigma(t) = (x^i(t))$
is a geodesic of $S$ if and only if it satisfies the equation
\begin{equation}
  \label{eq:sode}
  \frac{d^2 x^i}{dt^2} +2G^i\Big(x ,\frac{dx}{dt}\Big)=0.
\end{equation}

\ok{A nonlinear connection is described by a supplemental $n$-dimensional
  distribution to the vertical distribution, denoted as
  $\H\colon u \in \tm \rightarrow \H_u\subset T_u(\tm)$. For every
  $u \in \tm$, we have}
\begin{equation}
T_u(\tm) = \H_u \oplus \V_u.
\end{equation}
Every spray $S$ induces a canonical nonlinear connection
\ok{\cite{Grifone_1972}}  through the
corresponding horizontal and vertical projectors,
\begin{equation}
  \label{projectors}
  h=\frac{1}{2}  (Id + [J,S]), \,\,\,\,    v=\frac{1}{2}(Id - [J,S]).
\end{equation}
\ok{Equivalently, the canonical nonlinear connection defined  by a spray is expressed as an almost product structure $\Gamma = [J,S] = h - v$. A
  spray $S$ is horizontal with regard to the induced nonlinear connection,
  this means that $S = hS$. Moreover, the two projectors, $h$ and $v$, have
  the following local expressions }
\begin{displaymath}
  h=\frac{\delta}{\delta x^i}\otimes dx^i, \quad\quad
  v=\frac{\partial}{\partial y^i}\otimes \delta y^i,
\end{displaymath}
and the districutions are generated by the vector fields
\begin{displaymath}
  \frac{\delta}{\delta x^i}=\frac{\partial}{\partial
    x^i}-G^j_i(x,y)\frac{\partial}{\partial y^j},\quad
  \delta y^i=dy^i+G^j_i(x,y)dx^i, 
\end{displaymath}
where
\begin{math}
  G^j_i(x,y)=\frac{\partial G^j}{\partial y^i}.
\end{math}
\ok{If $X\in \mathfrak{X}(M)$, then $\mathcal{L}_X$ and $i_{X}$ stand for
  the Lie derivative with respect to $X$ and the interior product by $X$,
  respectively. $df$ represents the differential of $f\in C^\infty(M) $. A
  skew-symmetric $C^\infty(M)$-linear map
  $L:(\mathfrak{X}(M))^\ell\longrightarrow \mathfrak{X}(M)$ is a vector
  $\ell$-form on $M$. Each vector $\ell$-form $L$ defines two graded
  derivations of the Grassmann algebra of $M$, namely $i_L$ and $d_L$ as
  follows}
$$i_Lf=0, \,\,\,\, i_Ldf=df\circ L\,\,\,\,\,\, (f\in C^\infty(M)),$$
$$d_L:=[i_L,d]=i_L\circ d-(-1)^{\ell-1}di_L.$$
\ok{The curvature tensor $R$ of the nonlinear connection is 
\begin{equation}
  \label{eq:R}
  R = - \frac{1}{2} [h, h], 
\end{equation}}
and the Jacobi endomorphism \ok{\cite{Grifone_Muzsnay_2000}} is defined by
\begin{displaymath}
  \Phi=v\circ [S,h]=R^i_{\,\,j}\frac{\partial}{\partial y^i}\otimes dx^j=
  \left(2\frac{\partial G^i}{\partial x^j}-S(G^i_j)-G^i_kG^k_j \right)
  \frac{\partial}{\partial y^i}\otimes dx^j.
\end{displaymath}
The two curvature tensors are related by
\begin{displaymath}
  3R=[J,\Phi], \quad \Phi=i_SR. 
\end{displaymath}
For simplicity, we use the notations
\begin{displaymath}
  \delta_i:=\frac{\delta}{\delta x^i}, \quad
  \partial_i:=\frac{\partial}{\partial x^i}, \quad
  \dot{\partial}_i:=\frac{\partial}{\partial y^i}.
\end{displaymath}

\begin{definition} 
  \ok{A Finsler manifold of dimension $n$ is a pair $(M,F)$, where $M$ is a
    smooth manifold of dimension $n$, and $F$ is a continuous function
    $F: TM \to \mathbb{R}$ such that:
  \begin{itemize}[leftmargin=20pt]
  \item[a)] $F$ is smooth and strictly positive on $\T M$,
  \item[b)]$F$ is positively homogenous of degree $1$ in the directional
    argument $y${\em:} $\mathcal{L}_{\mathcal{C}} F=F$,
  \item[c)] The metric tensor
    $g_{ij}= \dot{\partial}_i \dot{\partial}_j E$ has rank $n$ on $\T M$,
    where $E:=\frac{1}{2}F^2$ is the energy function.
  \end{itemize}}
\end{definition}
Since the $2$-form $dd_JE$ is non-degenerate, the Euler-Lagrange equation
\begin{equation}
  \label{eq:EL}
  \omega_E:=i_Sdd_JE-d(E-\mathcal L_{{C}}E)=0
\end{equation}
uniquely determines a spray $S$ on $TM$.  This spray is called the
\emph{geodesic spray} of the Finsler function.  The $\omega_E$ is called
the Euler-Lagrange form associated to $S$ and $E$.

\bigskip

\subsection{Holonomy distribution and metrizability freedom}

 \begin{definition}[\cite{MZ_ELQ}]
   The \emph{holonomy distribution} $\hol$ of a spray $S$ is the
 distribution on $TM$ generated by the horizontal vector fields and their
 successive Lie-brackets, that is
  \begin{equation}
    \label{eq:14}
    \hol := \Bigl\langle \mathfrak X^h(TM)  \Bigr \rangle_{Lie}  \!
    =\Big\{[X_1,[\dots [X_{m-1},X_m]...]] \ \big| \ X_i
    \in \mathfrak X^h(TM) \Big\}
  \end{equation}
  where $\mathfrak X^h(TM)$ is the modules of horizontal vector fields.
\end{definition}

\ok{The \emph{parallel translation} along curves with respect to the
  canonical nonlinear connection associated to a spray $S$ can be
  introduced through horizontal lifts.  Let $c\colon [0,1] \to M$ be a
  piecewise smooth curve such that $c(0)=p$ and $c(1)=q$, and let $c^h$ be
  a horizontal lift of the curve $c$ (that is $\pi \circ c^h = c$ and
  $\dot c^h(t)\in \mathcal H_{c^h(t)}$). The parallel translation
  $\tau\colon T_pM\to T_qM$ along $c$ is defined as follows: if $c^h(0)=v$
  and $c^h(1)=w$, then $\tau(v)=w$.}

\begin{definition}
  \ok{Let $S$ be a spray. A function $E \in C^\infty(TM)$ is called a
    \emph{holonomy invariant function}, if it is invariant with respect to
    parallel translation induced by the associated canonical nonlinear
    connection to $S$.  That is, we have $E(\tau(v))=E(v)$, where $v\in TM$
    and $\tau$ is any parallel translation.} The set of holonomy invariant
  functions is  denoted by $\chol$.
\end{definition}
Since the parallel translations can be interpreted as travelling along
horizontal lift of curves \cite{MZ-Elgendi-1}, one can characterize the
element of $\chol$ as as functions with vanishing horizontal
derivatives. It follows that 
\begin{equation}
  \label{eq:hol}
  \chol =
  \left\{
    E \in C^\infty(\TM) \ | \ \mathcal{L}_X E=0, \ X \in \hol
  \right\}.
\end{equation}

\begin{definition}
  \ok{Suppose $S$ is a spray on a manifold $M$.  If there is a Finsler
    function $F$ such that its geodesic spray is $S$, then $S$ is called
    \emph{Finsler metrizable}.}
\end{definition}
Let us denote by $\mathcal F_S$ the set of Finsler function $F$ generating
$S$ as geodesic spray.  Then, we have
\begin{equation}
  \label{eq:mathcal_F}
  F\in \mathcal F_S \qquad \Longleftrightarrow \qquad
  E=\tfrac{1}{2}F^2 \in \chol
\end{equation}
that is $F$ is a Finsler function of $S$ if and only if the energy function
associated is a $2$-homogenous regular element of $\chol$.

The problems of how many essentially different Finsler metric can be
associated with a spray, and how to determine this number in terms of
geometric quantities were considered in \cite{MZ-Elgendi-1}.  In the case
when the holonomy distribution \eqref{eq:14} of a spray $S$ is regular,
then the metrizability freedom $\v (\in \mathbb N)$ can be calculated by
the following

\begin{theorem*}(\cite[Theorem 4.4]{MZ-Elgendi-1}) Let $S$ be a metrizable
  spray with regular holonomy distribution $\hol$. Then the metrizability
  freedom can be calculated as \ok{$\v=\mathrm{codim}(\hol)$}.
\end{theorem*}
In the case when the metrizability freedom of $S$ is $\v \geq 1$, then for
every $v_0\in \TM$ there exists a neighbourhood $U\subset \TM$ and
  functionally independent element
$E_1, \dots ,E_{\smallv}$ of $\chol$ on $U$ such that any $E\in \chol$ can
be expressed as
\begin{displaymath}
  \hphantom{\qquad \forall \ v \in U.}
  E(v)=\varphi\big(E_1(v), \dots, E_{\smallv} (v)\big), \qquad \forall \ v
  \in  U,
\end{displaymath}
with some function $\varphi\colon \R^{\smallv}\to \R$. We also remark, that
in that case, since $\hol$ is generated by horizontal vector fields
  \ok{and their Lie brackets}, it contains $\H$,
therefore
\begin{equation}
  \label{eq:holonomy_dist}
  \hol = \H \oplus \V_{\hol},
\end{equation}
where $\V_{\hol}$ denotes the vertical part of $\hol$. Since
$\dim(\H)=n$ we get
\begin{equation}
  \label{eq:vert_part_hol}
  \dim \V_{\hol} = n - \v.
\end{equation}

\bigskip

\section{Geodesic invariant functions}

\medskip

\begin{definition}
  Let $S$ be a spray on $M$. Then $\P\in C^\infty(\T M)$ is called a
  geodesic invariant function, if for any geodesics $c(t)$ of $S$ it
  satisfies $P(c'(t))\equiv const$.
\end{definition}
Obviously, for a given spray $S$ the function $P\in C^\infty(\T M)$ is a
geodesic invariant function if and only if
\begin{equation}
  \label{eq:p_s}
  \lie{S}\P=0,
\end{equation}
that is $\P$ is a first integral of $S$ \cite{BucataruConstCret}. In that
spirit we can call such a function an $S$-invariant function, referring
also to the spray determining the geodesic structure.  We remark that $\P$
is constant along $S$ if and only if \ok{the dynamical covariant derivative
  of $\P$ vanishes, see for example \cite{Bucataru1}}.

\ok{As the results of \cite{Bucataru1} and \cite{MZ-Elgendi-2} show,
  certain geometric distributions associated to sprays and their
  deformation can play a central role in the investigation of their
  metrizability property.  This is why, motivated by \cite{MZ-Elgendi-2},
  for further computation and analysis we introduce} a decomposition of the
horizontal (resp.~the vertical) distributions adapted to an $S$-invariant
function $\P$, homogeneous of degree $k\neq 0$: we introduce the
endomorphisms
\begin{equation}
  \label{eq:h&v_n-1}
  \hp=h -\frac{d_J\P}{k\P}\otimes S, \qquad  \qquad
  \vp=v-\frac{d_v\P}{k\P}\otimes \C .
\end{equation}
and we set
\begin{equation}
  \label{eq:H_V_n_1}
  \Hp:=Im \, \hp, \quad \quad
  \Vp:=Im \, \vp.
\end{equation}
We have the following

\begin{lemma}
  \label{lemma:h_v_sub} \
  \begin{enumerate}[leftmargin=25pt]
  \item Properties of $\vp$ and $\Vp$: \vspace{3pt}
    \begin{enumerate}
    \item [i) \ ] $\ker (\vp)= \H \oplus \Span{C}$,
    \item [ii) \ ] $\im (\vp)= \Vp$ is an $(n-1)$-dimensional involutive
      subdistribution of $\V$,
    \item [iii) \ ] any $X\in \Vp$ is an infinitesimal symmetry of $\P$
      that is $\mathcal L_X \P=0$, \vspace{3pt}
    \item [iv) \ ] the vertical distribution have the decomposition
      $\V = \Vp\oplus Span\{\C\}.$
    \end{enumerate}
    \vspace{3pt}

  \item Properties of $\hp$ and $\Hp$: \vspace{3pt}
    \begin{enumerate}
    \item [i) \ ] $\ker (\hp)= \V \oplus \Span{S}$,
    \item [ii) \ ] $\im (\hp)=\Hp$ is an $(n-1)$-dimensional
      subdistribution of $\H$,
   
    \item [iii) \ ] the horizontal distribution have the decomposition
      $\H=\Hp\oplus Span\{S\},$
    \end{enumerate} \vspace{3pt}
  \item $J(\Hp) = \Vp$.
  \end{enumerate}
\end{lemma}

\smallskip

\begin{proof}
  We prove \emph{(1)} in detail.  The computations for \emph{(2)} are
  similar.

  \medskip

  \emph{ad i)} We note that $\H=\Ker v$, therefore $\H\subset \Ker
  \vp$. Moreover, if $V\in \ker \vp$ is vertical, then using $v(V)=V$ we
  get
  \begin{displaymath}
    \vp(V)=0 \quad \Longleftrightarrow  \quad
    V=\frac{V(\P)}{k\P}\C,
  \end{displaymath}
  that is $V\in \Span {\C}$ and we get \emph{i)}.

  \medskip

  \emph{ad ii)} We introduce the simplified notation
  $\P_i:=\dot{\partial }_i \P$ and the vector fields
  \begin{subequations}
    \label{eq:h_i_v_i}
    \begin{alignat}{1}
      \label{eq:h_i_v_i_a}
      h_i&:=\hp(\delta_i) = \delta_i -\frac{\P_i}{k\P}S,
      \\
      \label{eq:h_i_v_i_b}
      v_i&:=\vp(\dot{\partial}_i) = \dot{\partial}_i-\frac{\P_i}{k\P}\C
    \end{alignat}
  \end{subequations}
  for $i = 1, \dots, n$. We get
  \begin{subequations}
    \label{eq:span_h_i_v_i}
    \begin{alignat}{1}
    \label{eq:span_h_i_v_i_a}
      \Hp&=\Span{h_1, \dots,h_n},
      \\
      \label{eq:span_h_i_v_i_b}
      \Vp&=\Span{v_1, \dots,v_n}.
    \end{alignat}
  \end{subequations}
  We note that the vector fields in \eqref{eq:span_h_i_v_i_a} (resp.~in
  \eqref{eq:span_h_i_v_i_b}) are not independent since $y^i h_i =0$
  (resp.~$y^i v_i =0$).  Because the $k$-homogeneity property of $\P$ (and
  the $(k-1)$-homogeneity property of $\P_i$) for any $v_i,v_j\in \Vp$,
  their Lie bracket is
  \begin{displaymath}
    [v_i,v_j]
    = \Big[\dot{\partial}_i-\frac{\P_i}{k\P}y^k\dot{\partial }_k, \
    \dot{\partial }_j - \frac{\P_j}{k\P}y^\ell\dot{\partial}_\ell\Big]
    =\frac{\P_i}{k\P}\dot{\partial }_j-\frac{\P_j}{k\P}\dot{\partial
    }_i=\frac{\P_i}{k\P}v_j-\frac{\P_j}{k\P}v_i
  \end{displaymath}
  and hence from \eqref{eq:span_h_i_v_i_b} we get that $[v_i,v_j]\in \Vp$
  hence $\Vp$ is involutive.

  \medskip
  
  \emph{ad iii)} One can check that the generators
  \eqref{eq:span_h_i_v_i_b} of the distribution are infinitesimal symmetry
  of $\P$. Indeed, using Euler's theorem of the homogeneous functions we
  get for the $k$-homogeneous $\P$:
  \begin{equation}
    \label{eq:C_P}
    \mathcal L_{C}\P=k\P,
  \end{equation}
  and therefore
  \begin{equation}
    \label{eq:L_v}
    \mathcal L_{v_i} \P = \dot{\partial}_i(\P)-\frac{\P_i}{k\P}\C(\P)
    =\P_i-\frac{\P_i}{k\P}k\P =0.
  \end{equation}

  \medskip

  \emph{ad iv)} Supposing $\C \in \Vp$ we get form
  \eqref{eq:span_h_i_v_i_b} that $\C=C^i v_i$ with some coefficients
  $C^i$.  Solving this equation, since $\C(\P)=k\P$ and $v_i(\P)=0$, we
  find that $\C(\P)=C^i v_i(\P)=0$ which is a contradiction.

  \bigskip

  For \emph{3)}, we note that for the generators \eqref{eq:h_i_v_i_a} of
  \eqref{eq:span_h_i_v_i_a} and \eqref{eq:h_i_v_i_b} of
  \eqref{eq:span_h_i_v_i_b}, we get
  \begin{equation}
    J h_i=  J\delta_i -\frac{\P_i}{k\P} J\!S= \dot{\partial }_i
    -\frac{\P_i}{k\P}\C  = v_i,
   \end{equation}
   $i = 1, \dots, n$, and this proves \emph{3)}.
\end{proof}

From Lemma  \ref{lemma:h_v_sub} we get the following

\begin{corollary}
  For a given spray $S$ on $TM$, then any non trivial $S$-invariant
  function $\P\in C^\infty (\T M)$ and homogeneous of degree $k\neq 0$
  gives rise to the direct sum decomposition
  \eqref{eq:TTM_decomposition}. Moreover, if $\P$ is constant along $\Hp$,
  then we have also \eqref{eq:ker_dP_bis}.
\end{corollary}
We have the following

\begin{proposition}
  \label{prop:hol_in_Vp}
  Let $(M,F)$ be a Finsler manifold with geodesic spray $S$.  If $\P$ is a
  $k$-homogeneous holonomy invariant function with $k\neq 0$, then
  \begin{equation}
    \label{eq:V_S_Vp}
    \V_{\hol}\subseteq \Vp.
  \end{equation}
\end{proposition}

\begin{proof}
 Assume that $\P$ is a $k$-homogeneous holonomy invariant
  function with $k\neq 0$, then $\P \in \chol$, and according to
  \eqref{eq:hol}, we have
  \ok{$\V_{\hol} \subseteq \hol \subseteq \mathrm{Ker} \, d \P$}.  It
  follows that
  \begin{displaymath}
    \ok{\V_{\hol} \subseteq   \V \cap \mathrm{Ker}\, d\P = \Vp,}
  \end{displaymath}
  where \ok{we use the notation \eqref{eq:H_V_n_1}.}
\end{proof}

\begin{remark}
  \ok{Let $(M,F)$ be a Finsler manifold with geodesic spray $S$.  If
    $\P$ is a $k$-homogeneous $S$-invariant  (but not necessarily
      holonomy invariant)  function with $k\neq 0$ and
    $\V_{\hol} \subseteq \V_\P$, then $d_hd_h \P=0$.}
\end{remark}

\begin{proof}
  \ok{We note that, since $\P$ is not necessarily a holonomy invariant
  function, we do not have $d_h \P =0$. However, the image of the curvature
  tensor $R$ is in the holonomy distribution.  If $\V_{\hol}\subseteq \Vp$,
  then $d_R\P=0$. On the other hand, using \eqref{eq:R} and the properties
  $d_{[h,h]}=[d_h,d_h]$ and 
  \begin{displaymath}
    [d_h,d_h]= d_hd_h-(-1)d_hd_h=2d_hd_h,
  \end{displaymath}
  we have 
  \begin{displaymath}
    d_hd_h\P=    \tfrac{1}{2}d_{[h,h]}\P =     - d_R\P=    0,
  \end{displaymath}
  which shows the statement of the remark.}
\end{proof}

It should be noted that in the generic case the holonomy distribution of a
spray is the $2n$-dimensional distribution $TTM$ and the metrizability
freedom is $\v=0$. For $\v=1$ we get the following

\begin{theorem}
  \label{thm:metric_freedom_1}
  Let $S$ be a given spray metrizability freedom $\v=1$, that is
  (essentially) uniquely metrizable by a Finsler function $F$. Then, for
  any $1$-homogeneous $S$-invariant function $\P$, we have
  $\V_{\hol}= \Vp$ if and only if $F=c\P$ where
  $c \in \mathbb{R}\setminus\halmaz{0}$.
\end{theorem}
\begin{proof}
  Since the metrizability freedom of $S$ is $1$, then by
  \cite{MZ-Elgendi-1} the codimension of $\hol$ is one. That is, the
  dimension of $\hol$ is $2n-1$ and by the fact that the dimension of
  $\H_{\hol}$ is $n$, we can conclude that the dimension of
  $\V_{\hol}=n-1$.

  Assume that $F=c\P$, then $\P$ is holonomy invariant 1-homogenous
  function. From \ok{Proposition \ref{prop:hol_in_Vp}} we have
  $\V_{\hol}\subseteq \Vp$. Since the dimension of both spaces is $n-1$, we
  get their equality.

  Conversely, assume that $\V_{\hol}= \Vp$, then
  \begin{displaymath}
    d_{\vp} F=0\Longrightarrow   d_vF-\frac{d_v\P}{\P}d_\C F=0.  
  \end{displaymath}
  Since $d_\C F=F$, then we have
  \begin{displaymath}
    d_vF-\frac{d_v\P}{\P}  F=0 \Longrightarrow \frac{d_v F}{F}
    =\frac{d_v\P}{\P}.
  \end{displaymath}
  Then, there exists a function $a(x)$ on $M$ such that
  $F=e^{a(x)}\P$. Now, since $\P$ is $S$-invariant then $\lie{S}\P=0$ and
  also we have $\lie{S} F=0$ and therefore $\lie{S} a(x)=0$. Locally, we
  obtain that
  \begin{displaymath}
    y^i\pa_i a(x)-2G^i\paa_i a(x)=0 \quad \Longrightarrow \quad y^i\pa_i a(x)=0.
  \end{displaymath}
  By differentiation with respect to $y^j$, we get $\pa_j a(x)=0$, that is
  $a(x)$ is constant function. Hence we get $F=c\P$.
\end{proof}

\begin{corollary}
  Let $(M,F)$ be a Finsler manifold with isotropic non vanishing curvature.
  Then, for any $1$-homogeneous $S$-invariant function $\P$, we have
  $\V_{\hol}= \Vp$ if and only if $F=c\P$, where $c$ is a non-zero
  constant.
\end{corollary}
 
\begin{proof}
  In the case, when the Finsler manifold has a non vanishing isotropic
  curvature, then by \cite{MZ-Elgendi-1}, the metrizability freedom of its
  geodesic spray is one.  Therefore the result follows by Theorem
  \ref{thm:metric_freedom_1}.
\end{proof}

The next theorem characterizes $\V_{\hol}$ and therefore $\hol$ as
intersection of distributions associated to geodesic invariant functions:

\begin{theorem}\label{hol_s=intersec.V_F}
  Let $S$ be a metrizable spray with regular holonomy distribution. Then,
  we have
  \begin{equation}
    \V_{\hol}=\underset{F \in \mathcal F_S}{\medcap} \V_{F}.
  \end{equation}
 \end{theorem}

 \begin{proof}
   \ok{Let us suppose that $S$ is a metrizable spray with
     regular holonomy distribution on an $n$-dimensional manifold $M$, and
     its metric freedom is $\v \ (\geq 1)$. According to \cite[Theorem
     4.4]{MZ-Elgendi-1}, we have $\mathrm{codim} (\hol)=\v$, or
     equivalently,
   \begin{equation}
     \label{eq:dim_Hol_S}
     \dim (\hol)=2n - \v,
   \end{equation}
   and at the neighbourhood of any $(x,y)\in TM$, there exists a set
   $\set{E_1, \dots E_{\v}}$ of energy functions associated with $S$ such
   that any energy function of $S$ can be locally written as a functional
   combination of $E_1, \dots E_{\v}$. It follows that the corresponding
   Finsler functions $\set{F_1, \dots F_{\v}}$ are functionally
   independent, and locally generating the set of Finsler functions of $S$,
   that is every Finsler function $F$ of $S$ can be written as a functional
   combination
   \begin{displaymath}
     F=\phi(F_1, \dots, F_{\smallv})
   \end{displaymath}
   with some 1-homogeneous function $\phi$.  It follows that
   \begin{equation}
     \label{eq:F_F_mu}
     \underset{F \in \mathcal F_S}{\medcap} \mathrm{Ker}(dF)
     =\medcap_{\mu=1}^{\smallv}\, \mathrm{Ker}(dF_\mu).     
   \end{equation}
   Since $\halmaz{F_1, \dots , F_{\smallv}}$ are functionally independent,
   their derivatives are linearly independent, therefore
   $\medcap_{\mu=1}^{\smallv} \, \mathrm{Ker} (dF_\mu)$ is characterized by
   $\v$ linearly independent equations in $TTM$. It follows that
  \begin{equation}
    \label{eq:codim_ker_df}
    \dim 
    \left(
      \medcap_{\mu=1}^{\smallv} \, \mathrm{Ker} (dF_\mu)
    \right) =
    \dim (TTM) - \v = 2n - \v. 
  \end{equation}
  Moreover, the functions $F_\mu$ are all holonomy invariant functions,
  therefore $\mathrm{Ker}(dF_\mu)$ contains the holonomy distribution for
  $\mu=1, \dots ,\v$, and as a consequence, their intersection
   \begin{math}
     \medcap_{\mu=1}^{\smallv} \mathrm{Ker}(dF_\mu)
   \end{math}
   also contains $\hol$. Since the dimension of the intersection
   \eqref{eq:codim_ker_df} and the dimension of the holonomy distribution
   \eqref{eq:dim_Hol_S} are equal, we get
   \begin{equation}
      \label{eq:Hol_S_subset_Ker}
     \hol  =  \medcap_{\mu=1}^{\smallv} \mathrm{Ker}(dF_\mu).
   \end{equation}
   Using the vertical projection for \eqref{eq:Hol_S_subset_Ker} we get
   \begin{displaymath}
     \label{eq:v_hol_cap}
     \begin{aligned}[t]
       \V_{\hol}
       & =v 
       \left(
       \hol
       \right)
        \stackrel{\eqref{eq:Hol_S_subset_Ker}}{=}
         v 
         \left(
         \bigcap_{\mu=1}^{\smallv} \mathrm{Ker}(dF_\mu)
         \right)       \stackrel{\eqref{eq:F_F_mu}}{=}
       \\
       & =
         v 
         \left(
         \underset{F \in \mathcal F_S}{\medcap} \mathrm{Ker}(dF)
         \right)
         =
         \underset{F \in \mathcal F_S}{\medcap} v 
         \left(
         \mathrm{Ker}(dF)
         \right)
         =
         \underset{F \in \mathcal F_S}{\medcap} \V_{F}
     \end{aligned}
   \end{displaymath}
   showing the statement of the theorem.}
\end{proof}

\begin{corollary}
  Let $S$ be a metrizable spray by a Finsler function $F$.
  Then, $\V_{\hol}= \V_{F}$ if and only if the metrizability freedom of
  $S$ is $\v=1$.
\end{corollary}

\begin{theorem} 
  Let $F$ be a Finsler function and $S$ its geodesic spray. Then if $\P$ is
  a $1$-homogeneous nontrivial $\V_{F}$-invariant function, then it is
  regular. Moreover, if $\P$ is $S$-invariant then $\P=cF$ with some
  constant $c\in \mathbb R$.
\end{theorem}
We remark that the theorem shows that the $S$-invariant and
$\V_{F}$-invariant property is essentially characterizing the Finsler
function associated to $S$.

\begin{proof}
  Let $\P$ be a $1$-homogeneous $\V_{F}$-invariant function. It follows
  that it satisfies the the system
  \begin{displaymath}
    d_{X}\P=0, \quad \forall X\in\V_{F}.
  \end{displaymath}
  Then we have
  \begin{displaymath}
    d_{v_{F}}\P=d_v\P-\frac{d_vF}{F}  \P=0\Longrightarrow \frac{d_v F}{F}
    =\frac{d_v\P}{\P}.
  \end{displaymath}
  Then, there exists a function $a(x)$ on $M$ such that
  $F=e^{a(x)}\P$. Then $\P=e^{-a(x)}F$, and hence $\P$ inherits its
  regularity from the Finsler function $F$.

  Now, assume that $\P$ is $S$-invariant, then we have $\lie{S}\P=0$ and
  using the fact that $\lie{S}F=0$, we have
  \begin{displaymath}
    \lie{S}F=\lie{S}e^{a(x)}\P=e^{a(x)}\P\lie{S}a(x)=0.
  \end{displaymath}
  Then, we obtain that $y^i\pa_ia(x)=0$. But by differentiating with
  respect to $y^j$ variable, we get $\pa_ja(x)=0$. That is $a(x)=const$.
  Consequently, we get $F=c\P$.
  
\end{proof}

\bigskip

\section{Applications to the Landsberg surfaces}

\medskip

\begin{definition}
  A Finsler metric $F$ on a manifold $M$ is called a Berwald metric, if in
  any standard local coordinate system in $\TM$ the connection coefficients
  $G^i_j(x,y)$ are linear. A Finsler metric $F$ is called Landsberg metric,
  if the landsberg tensor with the components
  $L_{ijk}=-\frac{1}{2}F G^h_{ijk}\frac{\partial F}{\partial y^h}$ is
  identically zero.  
\end{definition}
The Berwald and the Landsberg type Finsler metrics are the most important
particular cases in Finler geometry: for Berwald metrics the associated
canonical connection is linear, and for Landsberg metric the parallel
transport with respect to the canonical connection preserves the metric
\cite{bao}. It is well known that all Berwald type Finsler metric is also
\ok{Landsbergian}, but the long-open, so called unicorn problem: is there a
Landsberg metric that is not Berwald?  In higher dimensions ($n\geq 3$),
there exist non-regular Landsberg metrics which are not Berwladian, for
more details, we refer to \cite{Elgendi-solutions,Shen_example}.  In
dimension two, L. Zhou \cite{Zhou} investigated a class of Landsberg
surfaces and he claimed that this class is not Bewaldian. Later in
\cite{Elgendi-Youssef}, it was shown that the class is, in fact,
Berwaldian.  Up to the best of our knowledge, there is no example of
non-Berwaldian Landsberg surface.

A Finsler function $F$ with the geodesic spray $S$ is said to be of scalar
flag curvature if there exists a function $K \in C^\infty(\T M)$ such that
the Jacobi endomorphism $\Phi$ of the geodesic spray $S$ is given by
\begin{equation}
  \Phi = K(F^2J-Fd_JF\otimes \C).
\end{equation}
Since the Jacobi endomorphism $\Phi$ of any Finsler surface is in the above
form, then it is clear that all Finsler surfaces are of scalar flag
curvature $K(x,y)$.  Also, since the curvature $R$ of a spray vanishes if
and only if the Jacobi endomorphism vanishes, then the curvature of any
Finsler surface vanishes if and only if $K$ vanishes.

\bigskip

\ok{Whenever the scalar curvature $K$ of the Finsler surface is
  nonvanishing we} will use the so called Berwald frame, introduced by
Berwald in \cite{Berwald}: it is a frame on $\TM$ canonically associated to
a 2-dimensional Finsler manifold and used to investigate projectively flat
2-dimensional Finsler manifolds.  \ok{We note that when the scalar
  curvature vanishes, the Berwald frame is not defined. For more detail, we
  refer, for instance, to \cite{Taha}.}

\begin{lemma}
  \cite{Bucataru-Ebtsam} Let $(M,F)$ be a Finslerian surface with the
  geodesic spray $S$ \ok{and of flag curvature $K \neq 0$}.  Then the
  Berwald frame $\{S, H, {\mathcal C}, V\}$ satisfies $J H=V$,
  \begin{subequations}
    \begin{alignat}{1}
      \label{Berwald_f2 SH} 
      [S, H]&= K V,
      \\
      \label{Berwald_f2 VS} 
      [S,V] &= -H,
      \\
      \label{Berwald_f2 HV} 
      [H,V]&=S+I H +S(I) V,
    \end{alignat}
  \end{subequations}
  and
  \begin{equation}
    \label{V(F)=0}
    H(F)= V(F)=0. 
  \end{equation}
  Moreover, the Bianchi's identity is given by \cite[Proposition
  1.4]{Foulon}
  \begin{equation}
    \label{Eq:single_Bianchi}
    S^2(I)+V(K)+I\,K=0,
  \end{equation}
  where $K$ is the flag curvature and $I$ is the main scalar of $(M,F)$.
\end{lemma}
One can characterize the Berwald and Landsberg type Finler metrics in terms
of the main scalar:

\begin{lemma}
  \label{lem:landsberg_berwald}
  \cite{Matsumoto_2D} A Finsler surface $(M,F)$ is
  \begin{enumerate}
  \item Landsberg if and only if $S(I)=0.$
  \item Berwald if and only if $S(I)=0$ and $H(I)=0$.
  \end{enumerate}
\end{lemma}

\medskip

\begin{proposition}
  \label{Prop_K(x)}
  All Landsberg surfaces with basic flag curvature  are either
  Riemannian or have vanishing flag curvature.
\end{proposition}
\begin{proof}
  Let $(M,F)$ be a Landsberg surface with basic flag curvature, that is
  $K=K(x)$ is a function on the manifold $M$.  Then $V(K)=0$, and by using
  the fact that $S(I)=0$ together with \eqref{Eq:single_Bianchi}, we have
  \begin{displaymath}
    KI=0.
  \end{displaymath}
  Then we have either $K=0$ or $I=0$ and this completes the proof.
 \end{proof}

 \begin{proposition}
     For any  Landsberg surface $(M,F)$ 
  with non-vanishing curvature, we have
  \begin{equation}
\label{Eq:Bianchi_beta}
 \beta +I \, V(\beta)+H(I)+V^2(\beta)=0,
 \end{equation}
    where $\beta:=\frac{S(K_0)}{K_0}- S\left(\int_0^t I(t)dt \right)$, $K_0\in C^\infty(\T M)$, $V(K_0)=0$, $I$ is the main scalar of $(M,F)$ and the integration here is taken with respect to $V$. 
\end{proposition}

\begin{proof} 
  Assume that $(M,F)$ is Landsberg surface with non-vanishing $K$. We work
  on a neighbourhood of a point $(x_0, y_0) \in \TM$ where $F$ is
  regular. Then from Lemma \ref{lem:landsberg_berwald} we get that $S(I)=0$
  and hence $S^2(I)=S(S(I))=0$. Then, \eqref{Eq:single_Bianchi} has the
  form
 \begin{equation}
 \label{Eq:R;2}
 V(K)=-IK.
 \end{equation}
  Since $K\neq 0$, then we can write
 $$\frac{V(K)}{K}=-I.$$
 \ok{Using integration as in \cite{Sabau_Shimada} we obtain}
 \begin{equation}\label{Eq:R_mu}
   K=K_0 \exp{\left(-\int_0^t I(t)dt \right)},
 \end{equation}  
 where $K_0\in C^\infty(\T M)$ and $V(K_0)=0$. But since $K$ is homogeneous
 of degree $0$ and by the fact that $[\C,V]=0$, then $K_0$ must be
 homogeneous of degree $0$, that is, $\C(K_0)=0$.  That is, $V(K_0)=0$ and
 $\C(K_0)=0$, hence $K_0=K_0(x)$.
   
 Taking the fact that $S(I)=0$ into account, \eqref{Eq:R_mu} implies
 \begin{displaymath}
   S(K)=S(K_0) \exp\left(-\int_0^t I(t)dt \right)+ K S\left(-\int_0^t I(t)dt \right) =S(K_0)\frac{K}{K_0}+KS\left(-\int_0^t I(t)dt \right).
 \end{displaymath}
 From which we can write
   \begin{equation} 
     \label{Eq:R,1_mu,1}
   \frac{S(K)}{K}=\frac{S(K_0)}{K_0}+ S\left(-\int_0^t I(t)dt \right).
   \end{equation}
     Then, \eqref{Eq:R,1_mu,1} can be written in the form
\begin{equation}
  \label{Eq:R,1}
  S(K)=\beta K,
\end{equation}
where $\beta=\frac{S(K_0)}{K_0}+  S\left(-\int_0^t I(t)dt \right)$.    Applying $S$ on \eqref{Eq:R;2} and using
\eqref{Eq:R,1}, we have
\begin{equation}
  \label{Eq:R;2,1}
S(V(K))=-IS(K)=-\beta I K.
\end{equation}
Applying $V$ on \eqref{Eq:R,1} and using \eqref{Eq:R;2}, we have
\begin{equation}
\label{Eq:R,1;2}
V(S(K))=V(\beta)K+\beta V(K)=V(\beta)K-\beta I K.
\end{equation}
Now, by the property that $[V,S] = H$ \eqref{Berwald_f2 VS},
\eqref{Eq:R;2,1} and \eqref{Eq:R,1;2} we have
\begin{equation}
\label{Eq:R,2}
H(K)=V(\beta)K.
\end{equation}

From which together with \eqref{Eq:R;2}, we get
\begin{equation}
  \label{Eq:R,2;2}
  V(H(K))=V^2(\beta)K+V(\beta)V(K)=V^2(\beta)K-IK \ V(\beta).
\end{equation}
\begin{equation}
 \label{Eq:R;2,2}
 H(V(K))=-H(I)K-IH(K)=-H(I)K-IK \ V(\beta).
\end{equation}
Since $[H,V]K=H(V(K))-V(H(K))$ then by \eqref{Berwald_f2 HV}, \eqref{Eq:R,2;2} and \eqref{Eq:R;2,2} we have 
\begin{displaymath}
S(K)+I \,  H(K)= -K \, H(I)-K \, V^2(\beta)
\end{displaymath}
from which together with the fact that $K\neq 0$ and by \eqref{Eq:R,1},
\eqref{Eq:R,2}, we have
 \begin{equation*}
 \beta +I \, V(\beta)+H(I)+V^2(\beta)=0.
 \end{equation*}
 This completes the proof.
\end{proof}

As a consequence of the above proposition, we have the following \ok{result
  which is obtained by \cite{Ikeda} and \cite{Taha}, proved in a different
  way}.
\begin{theorem}\label{Th1:S(K)=0}
  Let $(M,F)$ be a Landsberg surface \ok{with non-zero flag curvature}. If
  the flag curvature is $S$-invariant, then the surface is
  \ok{Riemannian}.
\end{theorem}
\begin{proof}
  Let $(M,F)$ be a Landsberg surface with non-vanishing flag curvature $K$ and the property that $S(K)=0$. Then, by \eqref{Eq:R,1} we get that $\beta=0$ and therefore $V(\beta)=V^2(\beta)=0$. Now by   \eqref{Eq:Bianchi_beta}, we obtain that $H(I)=0$ and the surface is Berwaldian. Moreover,    by \eqref{Eq:R,2}, we have $H(K)=0$ and using the fact that $S(K)=0$,  \eqref{Berwald_f2 SH} implies
  \begin{displaymath}
    K \, V(K)=0,
  \end{displaymath}
  from which together with Proposition \ref{Prop_K(x)} the result follows. 
\end{proof}

\begin{theorem}
  \label{Prop_K_S-invariant}
  Let $(M,F)$ be a Landsberg surface with non-vanishing flag curvature $K$,
  then $K$ is $S$-invariant if and only if $K$ is constant. In this case,
  $F$ is Riemannian.
\end{theorem}

\begin{proof}
  Let $(M,F)$ be a surface with non-vanishing flag curvature $K$. It is
  obvious that if $K$ is constant then $S(K)=0$ and hence $K$ is
  $S$-invariant. Now, assume that $K$ is $S$-invariant, that is, $S(K)=0$.
  By \eqref{Eq:R,1}, $\beta=0$ and then by \eqref{Eq:R,2} we get that
  $H(K)=0$. Since $[S,H]=KV$, then $KV(K)=S(H(K))-H(S(K))=0$ and hence
  $V(K)=0$ since $K\neq0$. Moreover, $K$ is zero homogeneous in $y$, then
  $\C(K)=0$. Therefore, we have
  $$S(K)=0, \quad H(K)=0,\quad V(K)=0, \quad \C(K)=0$$
  which implies that $K$ is constant. Then, $F$ is Riemnnian by Theorem
  \ref{Th1:S(K)=0}.
\end{proof}

A smooth function $f$ on $\tm$ is said to be $H$-invariant if $H(f)=0$.
Let's end this work by the following result.
\begin{theorem}
  \label{Th:S(K)=0}
  Let $(M,F)$ be a Berwald surface \ok{with non-vanishing flag
    curvature}. Then, the flag curvature $K$ is $H$-invariant if and only
  if $K$ is constant.
\end{theorem}
\begin{proof}
  Let $(M,F)$ be a Berwald surface. If $K$ is constant then it is clear
  that $H(K)=0$ and hence it is $H$-invariant. Now, assume that
  $H(K)=0$. If $K=0$, then the proof is done. If $K\neq 0$, then by
  \eqref{Eq:R,2}, $V(\beta)=0$. Since the surface is Berwaldian, then
  $H(I)=0$. Therefore, by \eqref{Eq:Bianchi_beta}, $\beta=0$ and by
  \eqref{Eq:R,1}, we have $S(K)=0$.  Using \eqref{Berwald_f2 SH}, we get
  that $V(K)=0$ since $K\neq0$. Since $\C(K)=0$, we have
  $$S(K)=0, \quad H(K)=0,\quad V(K)=0, \quad \C(K)=0$$
  which means that $K$ is constant.
\end{proof}


 \bigskip \bigskip
 
\end{document}